\documentclass[11pt]{amsart}
\usepackage[utf8x]{inputenc}
\usepackage{amsthm,amssymb, amsmath, url} 
\usepackage{graphicx}
\usepackage[all]{xy}

\newtheorem{theorem}{Theorem}

\newtheorem{lemma}[theorem]{Lemma}

\title[Singularity of harmonic measure]{Singularity of harmonic measure for finitely supported random walks}
\author{Petr Kosenko} 
\email{pkosenko@math.ubc.ca}

\author{Giulio Tiozzo} 
\email{tiozzo@math.utoronto.ca, giulio.tiozzo@uniroma1.it}
\date{September 16, 2026}

\begin{document}

\begin{abstract}
In this paper we affirmatively resolve the singularity conjecture for finitely supported non-degenerate random walks on cocompact Fuchsian groups. The method we use is based on constructing a pair of geodesic currents, one for the Lebesgue measure and one for the random walk measure, and checking that they are different by using Fourier analysis on the boundary.
\end{abstract}

\maketitle
\section{Introduction}
In this paper we affirmatively resolve the singularity conjecture for finitely supported non-degenerate random walks on cocompact Fuchsian groups.
\begin{theorem}
	\label{main theorem}
Let $\Gamma < \textup{PSL}_2(\mathbb{R})$ be a cocompact Fuchsian group, and $\mu$ a finitely supported probability measure on $\Gamma$ 
such that the semigroup generated by its support is $\Gamma$.
Then the hitting measure $\nu$ for the random walk driven by $\mu$ is singular with respect to Lebesgue measure on $\partial \mathbb{H}$. 
\end{theorem}

We refer to \cite{DKN2009} and \cite[Conjecture]{kaimanovich2011matrix} for the statements of this conjecture in the literature. Combining Theorem \ref{main theorem} with the main result of \cite{GuivleJan90}, \cite{GuivLeJan1993}, the singularity conjecture is fully resolved for all Fuchsian groups of first kind, whose limit set is $\Lambda \Gamma = \partial \mathbb{H}$. Note that our proof does not require $\mu$ to be symmetric. 


\subsection*{History} The original motivation for the conjecture comes from the work of Furstenberg (\cite{furstenberg1963noncommuting}, \cite{furstenberg71}), who constructed random walks on lattices in semisimple Lie groups whose hitting measures are absolutely continuous at infinity. 
Those measures, however, are not finitely supported. The conjecture has been in the folklore for quite some time, being considered by several mathematicians, such as Y. Guivarc'h, V. Kaimanovich, F. Ledrappier. Its precise formulation is contained in Kaimanovich--LePrince 
\cite[Conjecture]{kaimanovich2011matrix} and Deroin--Kleptsyn--Navas~\cite{DKN2009}. 

Note moreover that Bourgain \cite{Bourgain2012} and B\'ar\'any--Pollicott--Simon \cite{MR2969625} construct finitely supported random walks on $SL_2(\mathbb{R})$ whose hitting measure is absolutely continuous, but in those cases the group generated by the support of the measure is not discrete. For more recent developments concerning stationary measures for dense subgroups of $SL_2(\mathbb{R})$ see Kittle \cite{https://doi.org/10.1112/plms.70072}.

For finite-covolume, non cocompact Fuchsian groups, the conjecture was already known due to Guivarc'h--LeJan, see \cite{GuivleJan90} and \cite{GuivLeJan1993}. Their approach uses the winding of the geodesic flow, and has been further developed by Gadre--Maher--Tiozzo \cite{GadreMaherTiozzo2015}, B\'enard \cite{benard2023winding}, and, in higher dimension, Randecker--Tiozzo \cite{2019arXiv190411581R}. Other complementary approaches are using the Green metric, starting with Blach\`ere--Ha\"issinsky--Mathieu \cite{blachere2011harmonic}, and then Gekhtman--Tiozzo \cite{MR4069238} and Kim--Zimmer \cite{KimZimmer2025}.  

Thus, the remaining open case was the case of cocompact $\Gamma$. In this case, singularity was proven for certain measures on cocompact Fuchsian groups with regular fundamental domains by Kosenko \cite{10.1093/imrn/rnaa213} and Carrasco--Lessa--Paquette \cite{Carrasco2019OnTS} and, in greater generality, for Fuchsian groups with centrally symmetric fundamental domains by Kosenko--Tiozzo \cite{kosenko_tiozzo_2022}. Another approach is via thermodynamic formalism, as in Garc\'ia--Lessa \cite{GarciaLessa2025} and Cantrell--Tanaka (\cite{CantrellTanaka2024}, \cite{CantrellTanakaManhattan2025}). 
Further progress was recently obtained by Kosenko~\cite{Kosenko2026} using a functional analytic approach and in dimension $3$ using Dehn filling by Bogachev--Kosenko--Tiozzo \cite{BogachevKosenkoTiozzo2025}. 

We direct the reader to the first author's survey \cite{survey} for the detailed history around the conjecture. 
 
\subsection*{Idea of the proof}
The proof can be formulated in the language of \emph{geodesic currents}, i.e. $\Gamma$-invariant measures 
on $\partial \mathbb{H} \times \partial \mathbb{H} \setminus \Delta$. The natural current associated to the Lebesgue measure is the so-called \emph{Liouville current}, while the natural current associated to the random walk is known as the \emph{Na\"im current}. It is based on 3 steps:

\begin{enumerate}
\item If $\nu$ is absolutely continuous  (hence so is the hitting measure for the reverse random walk $\check{\nu}$), then, by ergodicity of the geodesic flow on hyperbolic surfaces, the Na\"im current and the Liouville current are a multiple of each other.
\item If we set $A$ as a neighborhood (in the group) of the axis of a loxodromic element in $\Gamma$, then every bi-infinite sample path crossing from one side of $A$ to the other has to hit $A$. As a consequence, the Na\"im current factors as 
$$dJ(\xi, \eta) = \sum_{a, b \in A} K_{\check{\mu}}(a, \xi) (M_A)_{a, b} K_\mu(b, \eta) d\check{\nu}(\xi) d\nu(\eta),$$
where $K_{\check{\mu}} (\cdot, \cdot)$, $K_{\mu}( \cdot, \cdot)$ are the Martin kernels and $M_A$ the inverse of the Green operator restricted to $\ell^2(A)$. 
\item 
Under the assumption of absolute continuity, after performing a logarithmic change of coordinates, both currents define bounded integral operators on $L^2(\mathbb{R})$, and one can basically define the kernel of the current $J$ as 
$$\textup{ker}(J) := \left\{ \varphi \in L^2(\mathbb{R}) \ : \ \int \varphi(t) dJ(s, t)  = 0 \textup{ for a.e. } s \in \mathbb{R} \right\}.$$
We then prove that, in these coordinates, the kernel of the Liouville current is trivial, while, because of the factorization in (2), we show by Fourier analysis that the kernel of the Na\"im current is not trivial. This contradicts (1). 
\end{enumerate}



	

	
\subsection*{AI disclosure}
The proof was obtained with the help of ChatGPT. The paper was completely written by the authors, and we take full responsibility for the contents of this preprint.

\medskip
After the first draft of this manuscript was completed, we have discovered that an alternative proof has been independently obtained by 
Timoth\'ee B\'enard, see \cite{benard2026proofsingularityconjecturediscrete}. The proofs appear to be substantially different: while he uses winding statistics for the geodesic flow, described in cohomological terms, we use currents and Fourier analysis on the boundary circle. 
The proof in \cite{benard2026proofsingularityconjecturediscrete} is slightly more general, as it only requires that the group (rather than the semigroup) generated by the support of $\mu$ equals $\Gamma$. 

\section{Green functions and operators} 
Given a probability measure $\mu$ on $\Gamma$, we define the random walk driven by $\mu$ as 
$$w_n := g_1 \dots g_n$$
where the $(g_n)_{n \geq 1}$ are independent and identically distributed, with law $\mu$.
Let us assume that the support of $\mu$ generates $\Gamma$ as a semigroup. Let $P^n(x, y) := \mu^{*n}(x^{-1} y)$ be the probability that the walk starting from $x$ is at $y$ at the $n$th step.
Recall that the \emph{Green function} is defined as 
$$G(x, y) := \sum_{n \geq 0} P^n(x, y) \qquad \textup{for any }x, y \in \Gamma$$
and the Markov operator $P : \ell^2(\Gamma) \to \ell^2(\Gamma)$ is defined as 
$$Pf(x) := \sum_{g \in \Gamma} \mu(g) f(xg).$$
Since $\Gamma$ is non-amenable, by \cite{Day1964} the Markov operator has spectral radius $< 1$ on $\ell^2(\Gamma)$, hence we can define 
the \emph{Green operator} $G : \ell^2(\Gamma) \to \ell^2(\Gamma)$ as
$$G := (I - P)^{-1} = \sum_{n \geq 0} P^n.$$
Note that the relationship between the operator and the function is $G(x, y) = (G \delta_y)(x)$ for any $x, y \in \Gamma$.

We will also use the \emph{reflected measure} $\check{\mu}(g) := \mu(g^{-1})$, and add subscripts $\mu$ or $\check{\mu}$ to our objects when it is not clear to which measure they refer to. Note that by definition 
$$P_{\check{\mu}}(x, y) = P_\mu(y, x), \qquad G_{\check{\mu}}(x,y) = G_\mu(y, x) \qquad \textup{for any }x, y \in \Gamma,$$
hence also $P_{\check{\mu}} = P_\mu^*$ and $G_{\check{\mu}} = G_\mu^*.$
 
For any subset $A \subseteq \Gamma$, consider the restricted operator $G_A : \ell^2(A) \to \ell^2(A)$
defined as $G_A(f) := (G f)\vert_{A}$.

\begin{lemma} \label{bounded-inverse}
The operator $G_A$ has a bounded inverse $M_A := G_A^{-1} : \ell^2(A) \to \ell^2(A)$.
\end{lemma}

\begin{proof}
We need to prove that $G_A$ is coercive, which implies it is invertible with a bounded inverse. 

Let $w = Gv$, so that $v = w - Pw$. Then 
$$2 \textup{Re}\langle Gv, v \rangle = 2 \textup{Re} \langle w, w - Pw \rangle = \Vert w - P w \Vert^2 + \Vert w \Vert^2 - \Vert P w \Vert^2 \geq \Vert v \Vert^2$$
so 
$$ \textup{Re}\langle Gv, v \rangle \geq \frac{1}{2} \Vert v \Vert^2.$$
Since $\ell^2(A)$ embeds isometrically into $\ell^2(\Gamma)$, the previous estimate, taken with $v \in \ell^2(A)$, 
immediately also implies
$$\textup{Re}  \langle G_A v, v \rangle  \geq \frac{1}{2} \Vert v \Vert^2 \qquad \forall v \in \ell^2(A).$$
Hence $G_A$ is coercive and so (by Lax-Milgram) it is boundedly invertible. 
\end{proof}

\begin{lemma}
Let $A \subseteq \Gamma$ be a subset, and let $x, y \in \Gamma \setminus A$ such that every random walk path connecting two points $x, y$ hits $A$. Then we have the equality 
\begin{equation} \label{Green-equal}
G(x, y) = \sum_{a, b \in A} G(x, a) (M_A)_{a, b} G(b, y).
\end{equation}
\end{lemma}

\begin{proof}
Given $a \in A$, let $F(x, a) := \mathbb{P}_x(T_A < +\infty, w_{T_A} = a)$ denote the probability that the walk hits $A$ and the first location where it hits it is $a$. 
Then by the Markov property we have, for any $x \in \Gamma$ and any $b \in A$, 
\begin{equation} \label{first-stop-2}
G(x, b) = \sum_{a \in A} F(x, a)G(a, b).
\end{equation}
Let us define the operator $F : \ell^2(A) \to \ell^2(\Gamma)$ as $(Ff)(x) := \sum_{a \in A} f(a) F(x, a)$.
So we can rewrite \eqref{first-stop-2} as 
$$G \vert_{\ell^2(A)} = F G_A$$
and, since $G_A$ is invertible, 
\begin{equation} \label{first-stop-inverse}
F = G \vert_{\ell^2(A)} G_A^{-1}.
\end{equation}
Now, if every path connecting $x$ and $y$ hits $A$, we have, also from the Markov property, 
\begin{equation} \label{first-stop-1}
G(x, y) = \sum_{b \in A} F(x, b)G(b, y)
\end{equation}
hence, by applying \eqref{first-stop-inverse}, we obtain the desired equality
$$G(x, y) = \sum_{a, b \in A} G(x, a) (M_A)_{a, b} G(b, y).$$
\end{proof}

\section{Boundary currents}

We recall that for a random walk on $\Gamma$ driven by $\mu$, for every Borel $A \subseteq \partial \mathbb{H}$, the harmonic measure $\nu$ on the boundary $\partial \mathbb{H}$ is given by the probability that the random walk converges to a point in $A$. The convergence to the boundary is a classical fact, we refer to \cite{kaimanovich2000poisson}. Moreover, $\nu$ is a unique probability measure given by the $\mu$-stationarity condition 
\[
\nu = \sum_{g \in \Gamma} \mu(g) g_* \nu.
\]

Recall that a \emph{(geodesic) current} is a $\Gamma$-invariant Radon measure on $\partial \mathbb{H} \times \partial \mathbb{H} \setminus \Delta$, that is, on the space of oriented geodesics in $\mathbb{H}$.
One important example is the \emph{Liouville current}
$$dL(\xi, \eta) = \frac{d \xi \ d \eta}{(\xi - \eta)^2}.$$

Associated to a random walk driven by the probability measure $\mu$, we can also produce a current as follows. See \cite{CantrellTanaka2024} for details.
First, the \emph{Martin kernel} is defined as 
$$K_\mu(x, \xi) := \lim_{y \to \xi} \frac{G(x, y)}{G(e, y)} = \frac{d(x_* \nu)}{d \nu}(\xi) \quad \forall x \in \Gamma, \xi \in \partial \mathbb{H}.$$
Note that the limit exists as in this case the Martin boundary coincides with the Gromov boundary \cite{gouezel2015martin}. 
Moreover, the \emph{Na\"im kernel} is defined as 
$$\Theta(\xi, \eta) := \lim_{\stackrel{x \to \xi}{y\to \eta}} \frac{G(x, y)}{G(x, e)G(e, y)} \quad \forall \xi, \eta \in \partial \mathbb{H}.$$
Associated to it, we have the \emph{Na\"im current}
$$dJ(\xi, \eta) = \Theta(\xi, \eta) \ d{\check{\nu}}(\xi) \ d\nu(\eta).$$

\begin{lemma} \label{liouville}
If $\nu$ and $\check{\nu}$ are absolutely continuous, then there exists $c > 0$ such that the Na\"im current equals 
$$dJ(\xi, \eta) = c \frac{d \xi \ d \eta}{(\xi - \eta)^2}.$$
\end{lemma}

\begin{proof}
This is essentially \cite[Lemma 2.12]{CantrellTanakaManhattan2025}. Let $d L$ be the Liouville current, and recall the isomorphism $T^1 \mathbb{H} \cong (\partial \mathbb{H} \times \partial \mathbb{H} \setminus \Delta) \times \mathbb{R}$. For a vector $v \in T^1 \mathbb{H}$, let $v^+, v^- \in \partial \mathbb{H}$ be the forward and backward point of the geodesic tangent to $v$. If one considers the function $F:  T^1 \mathbb{H} \to \mathbb{R}$ defined by taking the Radon-Nikodym derivative $F(v) := \frac{d J}{d L}(v^-, v^+)$, then $F$ descends to a measurable function on $T^1(\mathbb{H}/\Gamma)$, invariant under the geodesic 
flow. By ergodicity of the geodesic flow, $F$ must be constant almost everywhere. 
\end{proof}

\section{Fourier analysis}

Given $\tau > 0$, let us define the translation operator $T : L^2(\mathbb{R}) \to L^2(\mathbb{R})$ 
$$T(f)(x) = f(x + \tau).$$ 
Let us also denote the Fourier transform as
$$\widehat{f}(\omega) :=  \int_{\mathbb{R}} f(t) e^{-2 \pi i \omega t} \ dt.$$
The fundamental fact in Fourier analysis that we are using is that translates of finitely many functions by a lattice do not 
span all of $L^2(\mathbb{R})$. 

\begin{lemma} \label{fourier}
Let $f_1, \dots, f_N$ be functions in $L^2(\mathbb{R})$, and let  
$$H:= \overline{\textup{Span}\{ T^n f_j \ : \ 1 \leq j \leq N, n \in \mathbb{Z} \}} \subseteq L^2(\mathbb{R}).$$
Then $H \neq L^2(\mathbb{R})$.
\end{lemma}

\begin{proof}
It follows from Bownik \cite{bownik2000structure}. After rescaling the real variable, let us assume that $\tau = 1$. 
By \cite[Prop. 1.5]{bownik2000structure}, there is a bijective correspondence between translation invariant closed subspaces of $L^2(\mathbb{R})$ 
and measurable range functions, i.e. measurable functions 
$$J : \mathbb{R}/\mathbb{Z} \to \{ \textup{closed subspaces of }\ell^2(\mathbb{Z}) \},$$ 
up to a.e. equivalence.

Defining the map $\mathcal{T} : L^2(\mathbb{R}) \to L^2(\mathbb{T}, \ell^2(\mathbb{Z}))$ as 
$$(\mathcal{T}f)(\omega) := \left( \widehat{f}(\omega + k)\right)_{k \in \mathbb{Z}}$$
then by \cite[Prop. 1.5]{bownik2000structure} the range function associated to $H$ is 
$$J_H(\omega) = \textup{Span}\{ \mathcal{T}f_1(\omega), \dots, \mathcal{T}f_N(\omega) \} \ \textup{for a.e. }\omega \in \mathbb{T}.$$
On the other hand, the full space $V = L^2(\mathbb{R})$ corresponds to the function 
$$J_V(\omega) = \ell^2(\mathbb{Z}) \ \textup{for a.e. }\omega \in \mathbb{T}.$$
Since $J_H$ and $J_V$ are not equal almost surely, then $H \neq V$. 
\end{proof}

\subsection*{Proof of the main theorem}
Let $h$ be a loxodromic element in $\Gamma$, and fix coordinates on $\mathbb{H}$ so that 
the two fixed points of $h$ are $0$ and $\infty$, and so that its action in these coordinates is $h(z) = e^\tau z$, for some $\tau > 0$. Let $\gamma$ be the geodesic connecting 
the two fixed points of $h$, and let us take $o = i \in \mathbb{H}$ as the base point. 
Then let 
$$L:= \sup_{g \in \textup{supp}(\mu)} d(o, go)$$ and define 
$$\mathcal{A} := \{ z \in \mathbb{H} \ : \ d(z, \gamma) \leq L +1 \}$$
the strip of width $L+1$ around $\gamma$. Note that $\mathbb{H} \setminus \mathcal{A}$ is disconnected, 
and, since the random walk has step size bounded by $L$, every sample path $(w_n o)$ from one side to the other of 
$\mathcal{A}$ has to intersect $\mathcal{A}$. 
Denote the group elements whose orbit point lies in $\mathcal{A}$ as
$$A := \{ g \in \Gamma \ : \ d(go, \gamma) \leq L +1 \}.$$ 
Note that, since the action of $\Gamma$ on $\mathbb{H}$ is properly discontinuous and the quotient $\mathcal{A}/\langle h \rangle$ is compact, there exists $b_1, \dots, b_N \in \Gamma$ such that 
$$A = \{ h^n b_j \ : \ 1 \leq j \leq N, n \in \mathbb{Z} \}.$$

Since both the harmonic and Lebesgue measure are quasi-invariant and ergodic, then if $\nu$ is not singular, then 
it must be absolutely continuous. 

Moreover, recall that by \cite[Theorem 1.5]{blachere2011harmonic} and \cite[Theorem 1.9]{KimZimmer2025}, the harmonic measure $\nu$ is absolutely continuous if and only if there exists $C > 0$ such that 
\begin{equation} \label{green} 
C^{-1} e^{-d(x o, y o)} \leq G(x, y) \leq C e^{-d(x o, y o)} \quad \textup{for all }x, y \in \Gamma
\end{equation}
and in this case, if we denote by $\lambda$ the \emph{visual measure} on $\partial \mathbb{H}$, which lies in the Lebesgue class with density $\frac{d \lambda}{dx} = \frac{1}{\pi(1 + x^2)}$, the density $\frac{d \nu}{d \lambda}$ is bounded almost everywhere. 

Thus, if $\nu$ is absolutely continuous, then \eqref{green} holds, which, since $G_{\check{\mu}}(x, y) = G_{\mu}(y, x)$, 
implies that the analogous estimate works for $G_{\check{\mu}}$, which in turn implies that $\check{\nu}$ is also absolutely continuous. 
Moreover, the densities $\frac{d \nu}{d \lambda}$, $\frac{d \check{\nu}}{d \lambda}$ are bounded almost surely.

Let us recall the structure of horofunctions for the hyperbolic plane. 
Given $\xi \in \partial \mathbb{H}$, let $r_\xi(t)$ denote the geodesic ray based at $o$ and converging to $\xi$. 
Recall now that the Busemann cocycle is 
$$B_\xi(o, z) := \lim_{t \to \infty} (t - d(r_\xi(t), z)) \qquad \textup{for any }\xi \in \partial \mathbb{H}, z \in \mathbb{H}$$ 
and that the function $b_\xi(t, z) := t - d(r_\xi(t), z)$ is increasing in $t$ for $t \geq 0$ and any fixed $z \in \mathbb{H}$.
In the upper half plane model, we can explicitly compute
$$B_\xi(i, z) = \log \left( \textup{Im}(z) \frac{| i-\xi |^2 }{|z-\xi|^2} \right).$$

\begin{lemma}
	\label{horofun}
	Let $h$ be as before, $\xi \in \partial \mathbb{H} \setminus \{0, \infty\}$, and $z \in \mathbb{H}$. Then we have
		$$B_\xi(i, h^n z) \le C_{\xi, z} - |n| \tau \quad \textup{ for any }n \in \mathbb{Z},$$ 
	where $\tau$ is the translation length of $h$ and $C_{\xi, z} > 0$ depends only on the choice of $\xi$ and $z \in \mathbb{H}$.
\end{lemma}
\begin{proof}
	This is a simple but instructive computation. Let $z \in \mathbb{H}$, and let us consider the case $n \to + \infty$ first. 
	In this case we estimate
	\[
	\begin{aligned}
		B_\xi(i, e^{n \tau} z) 
		&= \log \left( e^{n \tau} \textup{Im}(z) \frac{| i- \xi |^2 }{| z e^{n \tau} - \xi |^2} \right) = \log \left( \frac{\textup{Im}(z) | i- \xi |^2  }{|z  - \xi e^{- n \tau}| \cdot |z e^{n \tau} - \xi|} \right) = \\
		&
		= \log \left( \underbrace{\frac{\textup{Im}(z) | i- \xi |^2}{|z -  \xi e^{-n \tau}|} }_{\text{bounded}}  \right) + \log \left( \frac{1}{ |z e^{n \tau} - \xi| } \right)   \le C - n \tau.
	\end{aligned}
	\]
A similar computation works for $n \rightarrow - \infty$, as follows.
	\[
		B_\xi(i, e^{n \tau} z) = \log \left( e^{n \tau} \underbrace{\textup{Im}(z) \frac{| i- \xi |^2 }{| z e^{n \tau} - \xi |^2}}_{\text{bounded}} \right)  \le C -|n| \tau.
	\]
\end{proof}

\begin{lemma} \label{naim}
Assume that \eqref{green} holds. With $A$ as above, we have for any $\xi < 0 < \eta$ the equality 
$$\Theta(\xi, \eta) = \sum_{a, b \in A} K_{\check{\mu}}(a, \xi) (M_A)_{a, b} K_\mu(b, \eta).$$
\end{lemma}

\begin{proof}
Since the action is cocompact, every point in the limit set is conical. 
Hence, there exists $D > 0$ and sequences $(x_k), (y_k)$ in $\Gamma$ such that $x_k o \to \xi$, $y_k o \to \eta$, and 
$d(x_k o, r_\xi) \leq D$, $d(y_k o, r_\eta) \leq D$ for all $k$.
Then, for sufficiently large $k$ the points $x_k o, y_k o$ lie on opposite sides of the $L+1$-neighborhood of the axis of $h$, hence every sample path for the random walk from $x_k$ to $y_k$ hits $A$.

Thus, by \eqref{Green-equal}, 
$$G(x_k, y_k) = \sum_{a, b \in A} G(x_k, a) (M_A)_{a, b} G(b, y_k)$$
and, dividing all terms by $G(x_k, e) G(e, y_k)$, 
\begin{equation} \label{green-1}
\frac{G(x_k, y_k)}{G(x_k, e) G(e, y_k)} = \sum_{a, b \in A} \frac{G(x_k, a)}{G(x_k, e)} (M_A)_{a, b} \frac{G(b, y_k)}{G(e, y_k)}.
\end{equation}
Now, since both $x_k o$ and $y_k o$ converge to the boundary radially, we can estimate
\begin{align*}
u_k(a) := \frac{G(x_k, a)}{G(x_k, e)} & \leq C^2 e^{-d(x_k o, a o)+d(x_k o, o)} \leq C^2 e^{-d(r_\xi(t_k), a o)+d(r_\xi(t_k), o) +2 D} \\
& \leq C^2 e^{t_k - d(r_\xi(t_k), a o) + 2 D}  \leq C^2 e^{2 D} e^{B_\xi(o, a o)}
\end{align*}
hence, since $h^n(i) = e^{n \tau} i$, for any $j \in \{1, \dots, N\}$ and any $n \in \mathbb{Z}$ we estimate using Lemma \ref{horofun}
$$u_k(h^n b_j) =  \frac{G(x_k, h^n b_j )}{G(x_k, e)} \leq C^2 e^{2 D} e^{B_\xi(i, h^n b_j i)} \leq C_{\xi, j} e^{- |n| \tau}$$
and 
$$\sup_{k} \sum_{\stackrel{1 \leq j \leq N}{n\in \mathbb{Z}}} |u_k(h^n b_j) - u(h^n b_j)|^2 \leq \sum_{\stackrel{1 \leq j \leq N}{n\in \mathbb{Z}}}  4 C^2_{\xi, j} e^{-  2|n| \tau}$$
so, since $u_k(a) \to u(a)$ pointwise, if we set $u(a) :=  K_{\check{\mu}}(a, \xi)$, then $u_k \to u$ in $\ell^2(A)$, and analogously setting
$$v_k(b) := \frac{G(b, y_k)}{G(e, y_k)}, \qquad v(b) := K_\mu(b, \eta)$$
satisfies $v_k \to v$ in $\ell^2(A)$. Hence, by taking the limit as $k \to \infty$ in \eqref{green-1}, we obtain 
$$\langle u_k, M_A v_k \rangle \to \langle u, M_A v \rangle$$
so 
$$\Theta(\xi, \eta)  = \sum_{a, b \in A} K_{\check{\mu}}(a, \xi) (M_A)_{a, b} K_\mu(b, \eta) \qquad \textup{for any } \xi < 0 < \eta.$$
\end{proof}

\begin{proof}[Proof of Theorem \ref{main theorem}]
Let us assume that $\nu$ is absolutely continuous, so that, as seen above, also $\check{\nu}$ is absolutely continuous. 
From Lemma \ref{naim} and Lemma \ref{liouville} we have that 
\begin{equation} \label{equality-currents}
\sum_{a, b \in A} K_{\check{\mu}}(a, \xi) (M_A)_{a, b} K_\mu(b, \eta) d\check{\nu}(\xi) d\nu(\eta) = c \frac{ d\xi \ d \eta}{(\xi - \eta)^2}
\end{equation}
on $\mathcal{R} := (-\infty, 0) \times (0, +\infty)$. 
Now, consider the change of variables 
$$s = \ell^-(\xi) := \log(- \xi), \quad t = \ell^+(\eta) := \log(\eta)$$ 
and define for each $g \in \Gamma$ the Radon-Nikodym derivatives
$$k^{-}_g(s) = \frac{d(\ell^-_* g_* \check{\nu}\vert_{(-\infty, 0)} )}{d s}, \qquad k^{+}_g(t) = \frac{d(\ell^+_* g_* \nu \vert_{(0, \infty)} )}{d t}$$
hence we have by rewriting \eqref{equality-currents}
\begin{equation} \label{equality-currents-2}
\sum_{a, b \in A} k^-_a(s) (M_A)_{a, b} k^+_b(t) \ ds \ dt  = c \frac{ds \ dt }{4 \cosh^2((s-t)/2)}.
\end{equation}

Moreover, note that 
$$k^{\pm}_{h^n b_j}(t) = k^{\pm}_{b_j}(t - n \tau)\qquad \forall 1 \leq j \leq N, \forall n \in \mathbb{Z}$$
and, by using the fact that the density
$\frac{d \nu}{d \lambda}$, hence also $\frac{d (b_{j*} \nu)}{d \lambda}$, is bounded above, and that $\frac{d \lambda}{dx} = \frac{1}{\pi(1 + x^2)}$, we obtain the estimate
$$0 \leq k^{\pm}_{b_j}(t) \leq C_j e^{-|t|} \qquad \forall t \in \mathbb{R},$$
so $k^{\pm}_{b_j} \in L^2(\mathbb{R})$ for any $j$.
Thus, 
$$\sup_{t \in \mathbb{R}} \sum_{a \in A} k^{\pm}_a(t) < + \infty$$
and $(k_a^{\pm}(t))_{a \in A}$ lies in $\ell^1(A) \subseteq \ell^2(A)$ for any $t \in \mathbb{R}$. 

\begin{lemma}\label{bessel}
The formula 
\begin{equation} \label{Rdef}
(R\psi)_a = \int \psi(t) k_a^+(t) \ dt, \qquad a \in A
\end{equation}
defines a bounded operator $R : L^2(\mathbb{R}) \to \ell^2(A)$, and its adjoint $R^* : \ell^2(A) \to L^2(\mathbb{R})$ is 
\begin{equation} \label{adjointdef}
R^*f := \sum_{a \in A} f(a) k_a^+.
\end{equation}
\end{lemma}

\begin{proof}
Note that $(R\psi)_a$ is well defined since $k_a^+(\cdot) \in L^2(\mathbb{R})$ for any $a \in A$. 
Now, for any $\psi \in L^2(\mathbb{R})$ we have, using Cauchy-Schwarz and the fact that  $\int k_a^+(t) \ dt \leq 1$ 
as $k_a^+(t)$ is the density of a subprobability measure, 
\begin{align*}
\sum_{a \in A} \left| \int \psi(t) k_a^+(t) \ dt \right|^2 & 
\leq \sum_{a \in A} \int |\psi(t)|^2 k_a^{+}(t) \ dt \int k_a^+(t) \ dt \\
& \leq \sum_{a \in A} \int |\psi(t)|^2 k_a^{+}(t) \ dt \\
& \leq \int |\psi(t)|^2 \left( \sum_{a \in A}  k_a^{+}(t) \right) \ dt \\
& \leq C_1 \Vert \psi \Vert_2^2
\end{align*}
where $C_1 := \sup_{t \in \mathbb{R}} \left( \sum_{a \in A}  k_a^{+}(t) \right)$.
For the second claim, note that for any $a \in A$, if we denote by $\delta_a$ the characteristic function of the singleton $\{a\}$, 
$$\langle R^* \delta_a, \psi \rangle = \langle \delta_a, R \psi \rangle = \int k_a^+(t) \psi(t)  \ dt$$
for any $\psi \in L^2(\mathbb{R})$, hence $R^*\delta_a = k_a^+$ in $L^2(\mathbb{R})$, and formula \eqref{adjointdef} follows by linearity
and continuity of the adjoint. 
\end{proof} 

By Lemma~\ref{fourier} applied to $k^{+}_{b_1}, \dots, k^{+}_{b_N}$, there exists a nonzero function $\varphi \in H^\perp \subseteq L^2(\mathbb{R})$, hence 
\begin{equation} \label{int1}
\int k_a^+(t) \varphi(t) \ dt = 0 \qquad \textup{for any }a \in A, 
\end{equation}
hence $R \varphi = 0$. 
Fix a generic $s \in \mathbb{R}$, and define for any $b \in A$ 
$$q_b(s) := \sum_{a \in A} k_a^-(s) (M_A)_{ab}.$$
Note that $q := (q_b(s))_{b \in A}$ is a well-defined element in $\ell^2(A)$, since $(k^{-}_a(s))_{a \in A} \in \ell^2(A)$ and $M_A : \ell^2(A) \to \ell^2(A)$ is bounded by Lemma \ref{bounded-inverse}. 

By the formula for the adjoint \eqref{adjointdef} and equality \eqref{equality-currents-2} we obtain 
$$(R^*q)(t)  = \sum_b q_b(s) k_b^+(t) = c\ \mathcal{K}(s-t) \qquad \textup{in }L^2(\mathbb{R}, dt)$$
where $\mathcal{K}(x) := \frac{1}{4 \cosh^2(x/2)} \in L^1(\mathbb{R}) \cap L^2(\mathbb{R})$.

Hence, by definition of adjoint operator, 
\begin{equation*} 
\int c\  \mathcal{K}(s-t) \ \varphi(t) \ dt  = \langle R^* q, \varphi \rangle = \langle q, R \varphi \rangle = 0 \quad \textup{for almost every }s \in \mathbb{R}
\end{equation*}
hence also 
$$\mathcal{K} \star \varphi = 0 \qquad \textup{in }L^2(\mathbb{R}, ds).$$

Then, by taking the Fourier transform,   
$$0 = \widehat{\mathcal{K} \star \varphi } = \widehat{\mathcal{K}} \cdot \widehat{\varphi}   \quad \textup{in }L^2(\mathbb{R}, d \omega)$$
and since the Fourier transform of $\mathcal{K}$, namely 
$$\widehat{\mathcal{K}}(\omega) = \left\{ \begin{array}{ll} \frac{2 \pi^2 \omega}{\sinh(2 \pi^2 \omega)} & \textup{for }\omega \neq 0 \\
1 & \textup{for }\omega = 0, \end{array}\right.$$
nowhere vanishes, then $\widehat{\varphi} = 0$ a.e., which implies $\varphi = 0$ a.e. 
This is a contradiction, which concludes the proof. 
\end{proof}

\bibliographystyle{plain} 
\bibliography{singularity}

\end{document}